\documentclass{raex}
\usepackage{amsmath, amssymb, amsthm, amsfonts}

\begin{Author}
	\FirstName{Alex }\LastName{Kirtadze}
	\PostalAddress{ Department of Mathematics,~ Georgian Technical University, B. P.
0175,   Kostava St. 77, Tbilisi 75, Republic of  Georgia
}
	\Email{a.kirtadze  @ gtu.ge }
	\Thanks{The designated project has been fulfilled by financial support of Shota
Rustaveli National Science Foundation ( Grants:$\#$ GNFS 31/25,~$\#$ GNFS / FR  1165-100 / 14)  }
\end{Author}

\begin{Author}
	\FirstName{Gogi  }\LastName{Pantsulaia}
	\PostalAddress{ Department of Mathematics,~ Georgian Technical University, B. P.
0175,   Kostava St. 77, Tbilisi 75, Republic of  Georgia
}
	\Email{g.pantsulaia  @ gtu.ge }
	
\end{Author}

\begin{Author}
	\FirstName{Nino  }\LastName{Rusiashvili}
	\PostalAddress{ Department of Mathematics,~ Georgian Technical University, B. P.
0175,   Kostava St. 77, Tbilisi 75, Republic of  Georgia
}
	\Email{n.rusiashvili  @ gtu.ge }

\end{Author}

\begin{MathReviews}
	\primary{28C10}
	\primary{28A05}
	\secondary{26A42}
\end{MathReviews}

\begin{KeyWords}
\end{KeyWords}

\title{On uniform distribution for invariant extensions of the linear Lebesgue measure}

 \newtheorem{theorem}{Theorem}[section]

\newtheorem{lemma}[theorem]{Lemma}

\theoremstyle{definition}
\newtheorem{definition}[theorem]{Definition}

\theoremstyle{remark}
\newtheorem{remark}[theorem]{Remark}

\begin{document}

\maketitle

\begin{abstract}The concept  of uniform distribution in $[0,1]$ is extended for a certain  strictly separated maximal  (in the sense of cardinality)  family $(\lambda_t)_{t \in [0,1]}$  of  invariant extensions  of the linear Lebesgue measure $\lambda$   in $[0,1]$, and it is shown that the $\lambda_t^{\infty}$ measure of the set of all $\lambda_t$-uniformly distributed sequences is equal to $1$, where $\lambda_t^{\infty}$  denotes the infinite power of the measure $\lambda_t$. This  is an analogy  of Hlawka's (1956) theorem for  $\lambda_t$-uniformly
distributed sequences.  An analogy  of  Weyl's (1916)  theorem is  obtained  in the similar manner. 
\end{abstract}

\section{Introduction}

 The theory of uniform distribution is concerned
with the distribution of real numbers in the unit interval
$(0, 1)$ and  it's development started with Hermann Weyl's celebrated
paper \cite{Weyl}.  This theory gives a useful technique for numerical calculation exactly of the one-dimensional Riemann integral over $[0,1]$. More pricisely,
the sequence of real numbers $(x_n)_{n \in N} \in [0,1]^{\infty} $ is
uniformly distributed in $[0,1]$  if and only if for every real-valued Riemann integrable function $f$  on $[0, 1]$  the equality
$$
\lim_{ N \to \infty}\frac{\sum_{n=1}^Nf(x_n)}{N} =\int_{0}^{1} f(x) dx. \eqno (1.1)
$$
holds( see, for example \cite{KuNi74}, Corollary 1.1, p. 3).  Main corollaries of this assertion successfully were used in Diophantine approximations and have
applications to Monte-Carlo integration (cf. \cite{Weyl},\cite{Hardy1},\cite{Hardy2},\cite{KuNi74}).
Note that the set $S$ of all uniformly distributed sequences in $[0,1]$ viewed as a subset of $[0,1]^{\infty}$ has full $\lambda^{\infty}$-measure,
 where $\lambda^{\infty}$ denotes the infinite power of the linear lebesgue measure $\lambda$ in $[0,1]$( cf.  \cite{KuNi74}, Theorem 2.2(Hlawka), p. 183). For a fixed Lebesgue integrable function $f$ in $[0,1]$,  one can state a question asking {\it what is a maximal subset $S_f$ of $S$ each element of which can be used for calculation it's  Lebesgue integral over $[0,1]$ by the formula (1.1) and whether this subset has the full $\lambda^{\infty}$-measure}. This question has been  resolved  positively by  Kolmogorov's Strong Law of Large Numbers.
There naturally arises another question asking whether can be  developed analogous methodology for invariant extensions of the Lebesgue measure in $[0,1]$ and whether main results of the uniform distribution theory  will be  preserved in such a situation. In the present manuscript we consider this question for a certain strictly separated  \footnote{The family of probability measures $(\mu_i)_{i \in I}$ defined on the measure space $(E,S)$ is called strictly separated if there exists a partition $\{ E_i : i \in N ~\& E_i \in S\}$ of the set $E$  such that $\mu_i(E_i)=1$ for $ i \in I$.  } maximal (in the sense of cardinality) family of invariant extensions of the linear Lebesgue measure in $[0,1]$. In our investigations we essentially use the methodology developed in works \cite{Khar83}, \cite{Shiryaev1980}, \cite{KuNi74}.

The rest of the present paper it the following.

In Section 2  we consider some auxiliary facts from the theory of invariant extensions of the Lebesgue measure and from the probability theory.  In Section 3 we present our main results. Section 4 presents historical background  of the theory of  invariant extensions of the $n$-dimensional Lebesgue measure as well Haar measure in locally compact Hausdorff topological groups. In Section 5 we state a uniform distribution problem for invariant extensions of the Haar measure in locally compact Hausdorff topological groups.

\section{Some auxiliary notions and facts from the theory of invariant extensions of the Lebesgue  measure}

As usual, we  denote by $R$  the real axis  with usual metric and addition $''+''$ operation under which $R$ stands a locally compact $\sigma$-compact Hausdorff topological group.
We denote by $\lambda$ the linear Lebesgue measure in $R$.

\begin{lemma}(\cite{Khar83}, Lemma 6, p. 174)
Let  $K$ be a shift-invariant $\sigma$-ideal
of subsets of the real axis $R$ such
that
$$(\forall Z)(Z \in K \rightarrow \lambda_{\ast}(Z)=0),$$
where $\lambda_{\ast}$ denotes an inner measure defined by the linear Lebesgue measure $\lambda$.
Then the functional $\mu$  defined by
$$\mu((X \cup Z') \cup Z'')=\lambda(X),$$
where $X$ is Lebesgue measurable  subset of $R$ and $Z'$ and $Z''$are elements
of the $\sigma$-ideal $K$, ~is  a  shift-invariant extension of the Lebesgue
measure $\lambda$.
\end{lemma}
For the proof  of Lemma 2.1, see, e.g., \cite{Khar83}, \cite{Khar96}.

\begin{lemma} (\cite{Khar83}, Lemma 4, p. 164)
There exists a family $(X_i)_{i \in [0,1]}$ of subsets of the real axis
$R$ such that:

$1)~~ (\forall i)(\forall i')(i \in [0,1] \ \& \ i' \in [0,1] \ \& \ i
\neq i' \rightarrow X_i \cap X_{i'}=\emptyset );$

$2)~~ (\forall i)(\forall F)(i \in [0,1] ~\&~ (F$ is a closed subset
of the real axis $R$ with~$\lambda(F) > 0) \rightarrow
\mbox{card}(X_i \cap F)=c);$

$3)~~ (\forall I')(\forall g)(I' \subseteq [0,1] \ \& \ g \in R
\rightarrow card(g+(\bigcup_{i \in I'} X_i) \triangle (\bigcup_{i
\in I'} X_i)) < c ).$
\end{lemma}


\begin{lemma}
There exists a family
$(\mu_t)_{t \in [0,1]}$ of measures  defined on some shift-invariant $\sigma$-algebra $S(R)$
of subsets of the real axis $R$ such that:

$1)~~ (\forall t)(t \in [0,1] \rightarrow$ the measure $\mu_t$ is  a
shift-invariant extension of the linear Lebesgue measure $\lambda);$

$2)~~ (\forall t)(\forall t')(t \in [0,1] \ \& \ t' \in [0,1] \ \& \ t
\neq t' \rightarrow \mu_t$ and $\mu_{t'}$ are  orthogonal
\footnote{~~ $\mu_t$ and $\mu_{t'}$ are called orthogonal if there
exists $X \in S(R)$ such that $\mu_t(X)=0$ and $\mu_{t'}(R
\setminus X)=0.$}  measures. Moreover,  $\mu_t(R \setminus X_t)=0$ for each $t \in [0,1]$, where $(X_t)_{t \in [0,1]}$ comes from Lemma 2.2. $).$
\end{lemma}

\begin{proof} For arbitrary $t \in [0,1]$, we denote by  $K_t$  an shift-invariant $\sigma$-ideal generated by  the set $R \setminus X_t$. Then it is easy to verify that the
$\sigma$-ideal $K_t$ satisfies all conditions of Lemma 2.1. Let us denote by
$\overline{\mu}_t$ the shift-invariant extension of the Lebesgue measure  $\lambda$ produced
by the $\sigma$-ideal $K_t$. We obtain the family $(\overline{\mu}_t)_{t
\in [0,1]}$ of shift-invariant extensions of the Lebesgue measure
$\lambda$.

Denote by $S(R)$ ~the~shift-invariant $\sigma$-algebra of subsets of
the real axis  $R$, generated by the union
$$
    F(R) \cup  L(R) \cup \{ X_t  : t \in [0,1] \},
$$
where
$$
F(R)=\{ X : X \subseteq R~\&~\mbox{card}(X)<c\}
$$
and $L(R)$ denotes the class of all Lebesgue measurable subsets of the real axis $R$.

Also, assume that
$$(\forall t)(t \in [0,1] \rightarrow \mu_t=\overline{\mu}_t|_{S(R)}).$$

If we consider the family of shift-invariant measures
$(\mu_t)_{t \in [0,1] }$, we can easily conclude that this family satisfies all conditions
of Lemma 2.3.
\end{proof}
\begin{remark}
 Let consider the family $(\mu_t)_{t \in [0,1]}$ of shift-invariant extensions of the measure $\lambda$ which comes from Lemma 2.3. Let denote by $\lambda_t$ the restriction of the measure $\mu_t$ to the class
$$ S[0,1]:=\{  Y \cap [0,1] : Y \in S(R) \},$$
where $S(R)$ comes from Lemma 2.3. It is obvious that for each $t \in [0,1]$, the measure $\lambda_t$ is concentrated on the set $C_t=X_t \cap [0,1]$ provided that $\lambda_t([0,1] \setminus  C_t)=0$.
\end{remark}

The next proposition is useful for our further consideration.

\begin{lemma}( Kolmogorov Strong Law of Large Numbers, \cite{Shiryaev1980},Theorem 3, p.379) Let $(\Omega, \mathcal{S} ,P)$ be a probability space and $(\xi_k)_{k \in \mathbf{N}}$ be  a sequence of independent equally distributed random variables for which mathematical expectation $m$ of $\xi_1$ is finite. Then the following condition
$$
P(\{ \omega : \omega \in \Omega~\&~\lim_{n \to \infty}\frac{\sum_{k=1}^n\xi_k(\omega)}{n}=m\})=1
$$
holds.
\end{lemma}

\section{Uniformly  distribution  for invariant extensions of the  Lebesgue measure defined  by Remark 2.4}

Let consider the family of probability measures $(\lambda_t)_{t \in [0,1]}$ and the family $(C_t)_{t \in [0,1]}$ of subsets of $[0,1]$  which  come from Remark 2.4.

\begin{lemma}
   For $t \in [0,1]$, we denote by ${\bf L}([0,1],\lambda_t)$ the class of $\lambda_t$-integrable functions.
Then for $f \in {\bf L}([0,1],\lambda_t)$, we have
$$
\lambda_t^{\infty}(\{ (x_k)_{k \in N} : (x_k)_{k \in N} \in [0,1]^{\infty} ~\&~ \lim_{n \to \infty}\frac{\sum_{k=1}^nf(x_k)}{n}=\int_{[0,1]} f(x)d \lambda_t(x)\})=1.
$$
\end{lemma}

\begin{proof}  For fixed $t \in [0,1]$, we  set
$$(\Omega, ~{S} ,P)=(C_t^{\infty}, ~{F}(C_t)^{\infty}, \nu_t^{\infty}),$$
where

$i) ~~{F}(C_t)=\{ C_t \cap Y : Y \in S[0,1]\},$ where $S[0,1]$ comes from Remark 2.4.

$ii) ~ \nu_t= \lambda_t|_{~{F}(C_t)}$, where $\lambda_t|_{~{F}(C_t)} $ denotes restriction of the measure $\lambda_t$ to the sigma algebra $~{F}(C_t)$.

For $k \in N$ and $(x_k)_{k \in N} \in C_t^{\infty}$ we put $\xi_k((x_i)_{i \in N})=f(x_k)$. Then all conditions of  Lemma 2.5 are satisfied which implies that
$$\nu_t^{\infty}(\{ (x_k)_{k \in N} : (x_k)_{k \in N} \in C_t^{\infty} ~\&~ \lim_{n \to \infty}\frac{\sum_{k=1}^n\xi_k((x_i)_{i \in N})}{n}=
$$
$$
\int_{C_t^{\infty}} \xi_1((x_i)_{i \in N})d \nu_t^{\infty}((x_i)_{i \in N})\})=1,
$$
equivalently,
$$\nu_t^{\infty}(\{ (x_k)_{k \in N} : (x_k)_{k \in N} \in C_t^{\infty} ~\&~ \lim_{n \to \infty}\frac{\sum_{k=1}^nf(x_k)}{n}=\int_{C_t} f(x)d \nu_t(x)\})=1.
$$
The latter relation implies
$$\lambda_t^{\infty}(\{ (x_k)_{k \in N} : (x_k)_{k \in N} \in [0,1]^{\infty} ~\&~ \lim_{n \to \infty}\frac{\sum_{k=1}^nf(x_k)}{n}=\int_{[0,1]} f(x)d \lambda_t(x)\})\ge $$
$$\nu_t^{\infty}(\{ (x_k)_{k \in N} : (x_k)_{k \in N} \in C_t^{\infty} ~\&~ \lim_{n \to \infty}\frac{\sum_{k=1}^nf(x_k)}{n}=\int_{C_t} f(x)d \nu_t(x)\})=1.
$$
\end{proof}

\begin{definition}
A sequence of real
numbers $(x_k)_{k \in \mathbb{N}} \in [0,1]^{\infty}$ is said to be $\lambda$-uniformly distributed sequence (abbreviated $\lambda$-u.d.s.) if for each $c,d$ with $0 \le c
<d \le 1$  we have
$$
\lim_{n \to \infty}\frac{\#(\{ x_k : 1 \le k \le n\} \cap
[c,d])}{n}=d-c.\eqno (3.1)
$$
\end{definition}
We denote by $S$ the set of all real valued sequences from $[0,1]^{\infty}$ which are $\lambda$-u.d.s. It is well known that $(\{\alpha n\})_{n \in N} \in S$ for each irrational number $\alpha$, where  $\{\cdot\}$ denotes the fractional part of the real number(cf. \cite{KuNi74}, Exercise 1.12, p. 16).

\begin{definition}
A sequence of real
numbers $(x_k)_{k \in \mathbb{N}} \in R^{\infty}$ is said to be uniformly distributed module $1$ if the sequence it's fractional parts
$(\{x_k\})_{k \in \mathbb{N}}$ is $\lambda$-u.d.s.
\end{definition}
\begin{remark} It is obvious that $(x_k)_{k \in \mathbb{N}} \in (0,1)^{\infty}$  is uniformly distributed module $1$ if and only if $(x_k)_{k \in \mathbb{N}}$ is $\lambda$-u.d.s.
\end{remark}
\begin{definition}
A sequence of real
numbers $(x_k)_{k \in \mathbb{N}} \in [0,1]^{\infty}$ is said to be $\lambda_t$-uniformly distributed sequence (abbreviated $\lambda_t$-u.d.s.) if for each $c,d$ with $0 \le c
<d \le 1$  we have
$$
\lim_{n \to \infty}\frac{\#(\{ x_k : 1 \le k \le n\} \cap
[c,d]\cap C_t)}{n}=d-c.\eqno (3.2)
$$
\end{definition}
We denote by $S_t$ the set of all real valued sequences from $[0,1]^{\infty}$ which are  $\lambda_t$-u.d.s.

In order to construct  $\lambda_t$-u.d.s.  for each $t \in [0,1]$, we need the following lemma.

\begin{lemma}( \cite{KuNi74}, THEOREM 1.2, p. 3)  If the sequence $(x_n)_{n \in N}$  is u.d. mod 1, and
if $(y_n)_{n \in N}$ is a sequence with the property $\lim_{n \to \infty} (x_n-y_n)=\alpha$  for some real constant $\alpha$ ,
then $(y_n)_{n \in N}$  is u.d. mod 1.
\end{lemma}

\begin{theorem} For each $t \in [0,1]$, there exists $\lambda_t$-u.d.s.

\end{theorem}

\begin{proof}  Let consider  a sequence $(x_n)_{n \in N} \in (0,1)^{\infty}$  which is $\lambda$-u.d.s.  For each $n \in N$, we choose such an element $y_n$ from the set $C_t \cap (0, x_n) $  that $|x_n-y_n|<\frac{1}{n}$.  This we can do because $C_t$ is everywhere dense in $(0,1)$. Now it is obvious that  $\lim_{n \to \infty} (x_n-y_n)=0$. By Lemma 3.6 we deduce that  $(y_n)_{n \in N}$ is $\lambda$-u.d.s.  Let us show  that $(y_n)_{n \in N}$ is $\lambda_t$-u.d.s.  Indeed, since $y_k \in C_t$ for each $k \in N$  and  $(y_n)_{n \in N}$ is $\lambda$-u.d.s.,  for each $c,d$ with $0 \le c
<d \le 1$  we have
$$
\lim_{n \to \infty}\frac{\#(\{ y_k : 1 \le k \le n\} \cap
[c,d]\cap C_t)}{n}=
$$
$$
\lim_{n \to \infty}\frac{\#(\{ y_k : 1 \le k \le n\} \cap
[c,d])}{n}=d-c.\eqno (3.3)
$$
\end{proof}

\begin{theorem} For each $t \in [0,1]$,  $\lambda_t$-u.d.s.  is  $\lambda$-u.d.s..
\end{theorem}
\begin{proof} Let $(x_k)_{k \in N}$ be $\lambda_t$-u.d.s.
On the one hand,
for each $c,d$ with $0 \le c
<d \le 1$  we have
$$
\underline{\lim}_{n \to \infty}\frac{\#(\{ x_k : 1 \le k \le n\} \cap
[c,d])}{n}\ge
$$
$$
\lim_{n \to \infty}\frac{\#(\{ x_k : 1 \le k \le n\} \cap
[c,d]\cap C_t)}{n}=d-c.\eqno (3.4)
$$

Since $(x_k)_{k \in N}$ is $\lambda_t$-u.d.s., we have
$$
\lim_{n \to \infty}\frac{\#(\{ x_k : 1 \le k \le n\} \cap
[0,1] \cap C_t)}{n}=1.\eqno (3.5)
$$
It is obvious that
$$
\lim_{n \to \infty}\frac{\#(\{ x_k : 1 \le k \le n\} \cap
[0,1])}{n}=1.\eqno (3.6)
$$

The last two conditions implies  that

$$
\lim_{n \to \infty}\frac{\#(\{ x_k : 1 \le k \le n\} \cap
([0,1]\setminus C_t) )}{n}=
$$
$$
\lim_{n \to \infty}\frac{\#(\{ x_k : 1 \le k \le n\} \cap
[0,1])}{n}-
$$
$$
\lim_{n \to \infty}\frac{\#(\{ x_k : 1 \le k \le n\} \cap
 C_t)}{n}=1-1=0.\eqno (3.7)
$$

The last relation implies that for each $c,d$ with $0 \le c
<d \le 1$
$$
\lim_{n \to \infty}\frac{\#(\{ x_k : 1 \le k \le n\} \cap
[c,d] \cap ([0,1]\setminus C_t))}{n} \le
$$

$$\lim_{n \to \infty}\frac{\#(\{ x_k : 1 \le k \le n\} \cap
([0,1]\setminus  C_t))}{n}=0.\eqno (3.8)
$$
Finally, for each $c,d$ with $0 \le c
<d \le 1$ we get
$$
\overline{\lim}_{n \to \infty}\frac{\#(\{ x_k : 1 \le k \le n\} \cap
[c,d])}{n}\le
$$
$$
\lim_{n \to \infty}\frac{\#(\{ x_k : 1 \le k \le n\} \cap
[c,d]\cap C_t)}{n} +
$$
$$
\lim_{n \to \infty}\frac{\#(\{ x_k : 1 \le k \le n\} \cap
[c,d]\cap ([0,1] \setminus C_t))}{n}=
$$
$$
(d-c)+0=d-c.\eqno (3.9)
$$
This ends the proof of theorem.
\end{proof}

\begin{remark} Note that the converse to the result of Theorem 3.8 is not valid. Indeed, for fixed $t \in [0,1]$, let $(y_n)_{n \in N}$ be $\lambda_t$-u.d.s.  which comes from Theorem 3.7. By Theorem 3.8, $(y_n)_{n \in N}$ is $\lambda$-u.d.s. Let us show that $(y_n)_{n \in N}$ is not
$\lambda_s$-u.d.s. for each $s \in [0,1] \setminus \{t\}$. Indeed, since $y_k \in C_t$ for each $k \in N$, we deduce that $y_k \notin C_s$ for each $s \in [0,1] \setminus \{t\}$. The latter relation implies that for  each $s \in [0,1] \setminus \{t\}$ and for each $c,d$ with $0 \le c
<d \le 1$  we have
$$
\lim_{n \to \infty}\frac{\#(\{ y_k : 1 \le k \le n\} \cap
[c,d]\cap C_s)}{n}=0  < d-c.\eqno (3.10)
$$
\end{remark}

\begin{remark} For each $\lambda$-u.d.s.  $(y_n)_{n \in N}$ there exists a countable subset $T \subset [0,1]$ such that $(y_n)_{n \in N}$ is not $\lambda_t$-u.d.s for each $t \in [0,1] \setminus T$. Indeed, since $\{ C_t: t \in [0,1]\}$ is the partition of the $[0,1]$, for each $k \in N$ there exists a unique $t_k \in [0,1]$ such that $y_k \in C_{t_k}$. Now we can put $T=\cup_{k \in N}\{t_k\}$.
\end{remark}

\begin{theorem} There exists $\lambda$-u.d.s  which is not  $\lambda_t$-u.d.s. for each $t \in [0,1]$.
\end{theorem}
 \begin{proof} Let consider  a sequence $(x_n)_{n \in N} \in (0,1)^{\infty}$  which is $\lambda$-u.d.s. Since $\{ C_t: t \in [0,1]\}$ is the partition of the $[0,1]$, for each $k \in N$ there exists a unique $t_k \in [0,1]$ such that $y_k \in C_{t_k}$. Now we can put $T=\cup_{k \in N}\{t_k\}$.
 Let $S_0=\{s_1,s_2, \cdots\}$ be a countable subset of the set $[0,1]\setminus T$.
 For each $n \in N$, we choose such element $y_n$ from the set $C_{s_n} \cap (0, x_n) $  that $|x_n-y_n|<\frac{1}{n}$.  This we can do because $C_t$ is everywhere dense in $(0,1)$  for each $t \in [0,1]$. Now it is obvious that  $\lim_{n \to \infty} (x_n-y_n)=0$. By Lemma 3.6 we deduce that  $(y_n)_{n \in N}$ is $\lambda$-u.d.s.  Let us show  that $(y_n)_{n \in N}$ is not  $\lambda_t$-u.d.s. for each $t \in [0,1]$. This follows from the fact that
 $card(\{ y_n: n \in N\} \cap C_t) \le 1$ for each $t \in [0,1]$. By this reason for each $t \in [0,1]$ and  for each $c,d$ with $0 \le c
<d \le 1$  we have
$$
\lim_{n \to \infty}\frac{\#(\{ y_k : 1 \le k \le n\} \cap
[c,d]\cap C_t)}{n}\le
$$
$$
 \lim_{n \to \infty}\frac{1}{n}=0 < d-c.\eqno (3.11)
$$

\end{proof}

\begin{theorem} $S_i \cap S_j=\emptyset $ for each different $i,j \in [0,1]$.
\end{theorem}
\begin{proof} Assume the contrary and let $(x_k)_{k \in N} \in S_i \cap S_j$.
On the one hand,
for each $c,d$ with $0 \le c
<d \le 1$  we have
$$
\lim_{n \to \infty}\frac{\#(\{ x_k : 1 \le k \le n\} \cap
[c,d] \cap C_i)}{n}= d-c.\eqno (3.12)
$$
On the other hand,
for same $c,d$  we have
$$
\lim_{n \to \infty}\frac{\#(\{ x_k : 1 \le k \le n\} \cap
[c,d] \cap C_j)}{n}= d-c.           \eqno (3.13)
$$
By Theorem 3.8 we know that $(x_k)_{k \in N}$ is $\lambda$-u.d.s. which implies that
for same $c,d$  we have
$$
\lim_{n \to \infty}\frac{\#(\{ x_k : 1 \le k \le n\} \cap
[c,d])}{n}= d-c.\eqno (3.14)
$$

But $(3.14)$ is not possible because $C_i \cap C_j=\emptyset$ which implies

$$
d-c=\lim_{n \to \infty}\frac{\#(\{ x_k : 1 \le k \le n\} \cap
[c,d])}{n} \ge
$$
$$
\lim_{n \to \infty}\frac{\#(\{ x_k : 1 \le k \le n\} \cap
[c,d] \cap C_i)}{n} +
$$
$$
\lim_{n \to \infty}\frac{\#(\{ x_k : 1 \le k \le n\} \cap
[c,d] \cap C_j)}{n}=
$$
$$
(d-c)+(d-c)=2(d-c).\eqno (3.15)
$$
We get the contradiction and theorem is proved.
\end{proof}

We have the following version  of Hlawka's  theorem(cf. \cite{Hlawka56} ) for $\lambda_t$-uniformly distributed sequences.

\begin{theorem} For $t \in [0,1]$, we have $\lambda_t^{\infty}(S_t)=1$.
\end{theorem}
\begin{proof}Let $(f_k)_{k \in \mathbf{N}}$ be a countable subclass
of $L([0,1], \lambda_t)$  which defines a $\lambda_t$-uniform distribution on $[0,1]$.
\footnote{We say that a family  $(f_k)_{k \in \mathbf{N}}$ of
elements  of $L([0,1], \lambda_t)$  defines a
 $\lambda_t$-uniform distribution on $[0,1]$, if for each
$(x_n)_{n \in N} \in [0,1]^{\infty}$ the
validity of the condition $\lim_{N \to
\infty}\frac{1}{N}\sum_{n=1}^Nf_k(x_n)=\int_{[0,1]}
f_k(x)d \lambda_t(x)$ for $k \in N$ implies that $(x_n)_{n \in N}$ is
$\lambda_t$-u.d.s. Indicator
functions of sets $[c,d] \cap C_t$
with rational $c,d$  is an example of such a family.} For $k
\in N$, we set
$$
B_k=\{ (x_k)_{k \in {\bf N}} : (x_k)_{k \in {\bf N}} \in
[0,1]^{\infty} ~\& ~\lim_{N \to
\infty}\frac{1}{N}\sum_{n=1}^Nf_k(x_n)=\int_{[0,1]}
f_k(x)d \lambda_t x\}.
$$
By Lemma 3.1 we know that $\lambda_t^{\infty}(B_k)=1$ for $k \in
\mathbf{N},$ which implies $\lambda_t^{\infty}(\cap_{k \in
{\bf N}}B_k)=1$. Hence
$$
\lambda_t^{\infty}(\{ (x_k)_{k \in {\bf N}} : (x_k)_{k \in
{\bf N}} \in [0,1]^{\infty} ~\& ~(\forall k)(k \in {\bf N}
\rightarrow \lim_{N \to \infty}\frac{1}{N}\sum_{n=1}^N f_k(x_n)=
$$
$$
\int_{[0,1]} f_k(x)d \lambda_t(x))\})=1.
$$
The latter relation means that $\lambda_t^{\infty}$-almost every
elements  of $[0,1]^{\infty}$ is $\lambda_t$-u.d.s., equivalently, $\lambda_t^{\infty}(S_t)=1$.

\end{proof}

We have the following analogue of H.Weyl theorem ( cf.  \cite{Weyl} )  for  $\lambda_t$-uniformly distributed  sequences.

\begin{theorem} For $t \in [0,1]$, we put $C_t[0,1]=\{\tilde{h}(x)= h(x) \times \chi_{C_t}(x): h \in C[0,1]\}$. Then the sequence $(x_n)_{n \in N}$ is $\lambda_t$-u.d.s. if and only if the following condition
$$
\lim_{N \to \infty}\frac{\sum_{n=1}^N\tilde{h}(x_n)}{N}=\int_{[0,1]}\tilde{h}(x)d\lambda_t(x)\eqno (3.16)
$$
holds for each $\tilde{h} \in C_t[0,1]$.
\end{theorem}
\begin{proof}Let $(x_n)_{n \in N}$ be  $\lambda_t$-u.d.s. and let $f(x)= \sum_{i=1}^{k-1}d_i \chi_{[a_i,a_{i+1}[\cap C_t}(x)$ be a spatial step function on $[0,1]$ , where $0 = a_o < a_1 <  \cdots < a_k = 1.$  Then it follows from
(3.2) that for every such $f$ equation (3.16) holds. We assume now that $\tilde{f} \in C_t[0,1]$. Given any $\epsilon > 0$, there exist,
by the definition of the Riemann integral, two step functions, $f_1$ and $f_2$
say, such that $f_1(x) < f (x) < f_2(x)$ for all $x \in [0,1]$ and
$$\int_{[0,1]}(f_2(x) -f_1(x)) d\lambda(x) < \epsilon.$$
Then it  is obvious that
$f_1(x)\chi_{C_t}(x) < \tilde{f} (x) < f_2(x)\chi_{C_t}(x)$ for all $x \in [0,1]$ and
$$\int_{[0,1]}(f_2(x)\chi_{C_t}(x) -f_1(x)\chi_{C_t}(x)) d\lambda_t(x) = \int_{[0,1]}(f_2(x) -f_1(x)) d\lambda(x) < \epsilon.$$

Then we have the following chain of inequalities:

$$
\int_{[0,1]}\tilde{f}(x)d\lambda_t(x)-\epsilon \le \int_{[0,1]}f_1(x)\chi_{C_t}(x)d\lambda_t(x)=
$$
$$
\lim_{N \to \infty}\frac{\sum_{n=1}^Nf_1(x_n)\chi_{C_t}(x_n)}{N}\le
$$
$$
\underline{\lim}_{N \to \infty}\frac{\sum_{n=1}^N f(x_n)\chi_{C_t}(x_n)}{N}=
$$
$$
\underline{\lim}_{N \to \infty}\frac{\sum_{n=1}^N \tilde{f}(x_n)}{N} \le
$$
$$
\overline{\lim}_{N \to \infty}\frac{\sum_{n=1}^N \tilde{f}(x_n)}{N}=
$$
$$
\overline{\lim}_{N \to \infty}\frac{\sum_{n=1}^N f(x_n)\chi_{C_t}(x_n)}{N} \le
$$
$$
\lim_{N \to \infty}\frac{\sum_{n=1}^Nf_2(x_n)\chi_{C_t}(x_n)}{N}=
$$
$$
\int_{[0,1]}f_2(x)\chi_{C_t}(x)d\lambda_t(x)\le
$$
$$
\int_{[0,1]}\tilde{f}(x)d\lambda_t(x)+\epsilon.\eqno (3.18)
$$

So that in the case of a  function $\tilde{f}$  the relation (3.16) holds.

Conversely, let a sequence $(x_n)_{n \in N}$ be given, and suppose that (3.16) holds
for every $\tilde{f} \in C_t[0,1]$. Let $[a, b)$ be an arbitrary
subinterval of $[0,1]$. Given any $\epsilon > 0$, there exist two continuous functions, $g_1$ and $g_2$ say, such that
$g_l(x) < \chi_{[a,b)}(x) < g_2(x)$  for $x \in [0,1]$ and at the same time
$\int_{[0,1]}(g_2(x) - g_l(x)) d\lambda x < \epsilon $. Note that at the same time we have
$g_l(x)\chi_{C_t}(x) < \chi_{[a,b) \cap C_t}(x) < g_2(x)\chi_{C_t}(x)$  for $x \in [0,1]$ and
$\int_{[0,1]}(\tilde{g}_2(x) - \tilde{g}_l(x)) d\lambda_t x < \epsilon$.

Then we get

$$
 b - a - \epsilon  < \int_{[0,1]}g_2(x) d \lambda x - \epsilon < \int_{[0,1]}g_1(x) d \lambda x =
 $$
 $$
 \int_{[0,1]}\tilde{g}_1(x) d \lambda_t x = \lim_{N \to \infty}\frac{\sum_{n=1}^N\tilde{g}_1(x_n)}{N}\le
 $$
 $$
 \underline{\lim}_{N \to \infty}\frac{\#(\{ x_1, \cdots, x_N\}\cap [a,b[ \cap C_t) }{N} \le
 $$
 $$
  \overline{\lim}_{N \to \infty}\frac{\#(\{ x_1, \cdots, x_N\}\cap [a,b[ \cap C_t) }{N} \le
 $$
 $$
 \lim_{N \to \infty}\frac{\sum_{n=1}^N\tilde{g}_2(x_n)}{N}=
 $$
 $$
 \int_{[0,1]}\tilde{g}_2(x) d \lambda_t x =\int_{[0,1]}g_2(x) d \lambda x
 $$
 $$
  \le  \int_{[0,1]}g_1(x) d \lambda x +\epsilon \le
 $$
 $$
 \int_{[0,1]}\chi_{[a,b[}(x) d \lambda x +\epsilon= b-a +\epsilon.\eqno (3.19)
$$
 Since $\epsilon$ is arbitrarily small, we have (3.2).
 \end{proof}

\section{Historical Background  for invariant extensions of the Haar measure}

\subsection{On Waclaw Sierpiniski problem}

By Vitali's celebrate theorem about existence of the linear Lebesgue non-measurable subset has been  shown that the domain of the Lebesgue measure in $R$  differs from the power set of the real axis $R$.  In this context the following question was 
naturally appeared:

{\it ''How far can we extend Lebesgue
measure and what properties can such an extension
preserve?''}

In 1935 E. Marczewski, applied Sierpi\'{n}ski  construction of an almost invariant set $A$, obtained a proper invariant extension of the Lebesgue measure in which the extended $\sigma$-algebra had contained  new sets of positive finite measure.   In connection with this result, Waclaw Sierpiniski  in 1936  posed the
following

{\bf  Problem (Waclaw Sierpi\'{n}iski)} {\it Let $D_n$ denotes the group of all isometrical transformations of the $R^n$. Does there exist any maximal $D_n$ -invariant
measure?}

The first result in this direction was obtained by Andrzej Hulanicki \cite{Hulanicki1962}
as follows:

{\bf Proposition} ( Andrzej Hulanicki (1962))   {\it If the continuum $2^{\aleph_0}$ is not real valued measurable cardinal then there does not exist any maximal invariant extension of the Lebesgue measure.}
 
 This
result was also obtained independently  by S. S. Pkhakadze \cite{Pkhakadze1958}
using similar methods.

In 1977 A. B. KHarazishvili  got the same answer in the one-dimensional
case without any set-theoretical assumption
(see \cite{Khar77} ).

Finally, in 1982 Krzysztof Ciesielski
and Andrzej Pelc generalized Kharazishvili's result to all
$n$-dimensional Euclidean spaces (see \cite{Cies85}).

Following  Solovay \cite{Solovay1970}, if the system of axioms   ''$ZFC~\&$~There exists  inaccessible cardinal'' is consistent then   the systems of axioms
 '' $ZF ~\& CD~\& $	Every set of reals is Lebesgue  measurable'' is also consistent.  This result implies that the answer to
  Waclaw Sierpiniski's problem  is affirmative.

Taking   Solovay's result  on the one hand, and  Krzysztof Ciesielski
and Andrzej Pelc  (or Andrzej Hulanicki or Pkhakadze)  result on the  other  hand,  we deduce  that the  Waclaw Sierpiniski's question is not solvable within the theory  $ZF ~\& CD. $

\subsection{ On Lebesgue measure's   invariantly extension  methods  in $ZFC$ }

Now days  there exists  a reach methodology for a construction of invariant  extensions of the Lebesgue  measure in $R^n$ as well the Haar measure  in a locally compact Hausdorff  topological group. Let us briefly consider main of them.

{\bf Method I.}( Jankowka-Wiatr)  Following \cite{HulanickiRyll-Nardzewski1979}, the first idea of extending the Lebesgue measure in $R^n$ to a larger $\sigma$-algebra in such a what that it remains invariant under translations belongs to Jankowka-Wiatr who in 1928 observed that one can add new sets to the $\sigma$-ideal of sets of Lebesgue measure zero and still preserve the invariance of the extended measure.  This method can be described  as follows:

{\it   Let  $K$ be a  shift-invariant $\sigma$-ideal
in the $n$-dimensional  Euclidean  space   $R^n$ such
that
$$(\forall Z)(Z \in K \rightarrow m_{\ast}(Z)=0),$$
where $m_{\ast}$ denotes the inner measure defined by $n$-dimensional  Lebesgue measure $m$.
Then the functional $\overline{m}$  defined by
$$\overline{m}((X \cup Z') \setminus  Z'')=m(X),$$
where $X$ is a Borel subset of $R^n$ and $Z'$ and $Z''$are elements
of the $\sigma$-ideal $K$, ~is an $D_n$-invariant extension of the Lebesgue measure
measure $m$.}
\medskip

{\bf Method II.}(  E. Szpilrajn(E. Marczewski)) 
By using Sierpi\'{n}ski's decomposition $\{A,B\}$ of the  $R^2$,  E. Szpilrajn  noted that 
the following two conditions 

(i)  $card(A\ triandle (x+A))<c,~card(A\triangle (x+A))<c$ for each $x \in R^2$;~

(ii) $card(A \cap F)=card(B \cap F)=2^{\aleph_0}$ for each closed set $F \subseteq R^2$  whith  $m(F)>0$.

holds true.

Further,  he  constructed   a proper shift-invariant extension  $\overline{m}$ of the Lebesgue measure $m$ in $R^2$ 
as follows
$$\overline{m}((A \cap X )\cup (B \cap Y))=1/2(m(X)+m(y))$$
for $X,Y \in dom(m)$.

\medskip

{\bf Method III.} ( Oxtoby and Kakutani )  Some methods of combinatorial set theory have lately been
successfully used in  measure extension problem.
Among them, special mention
should be made of the method of constructing a maximal (in the
sense of cardinality) family of independent families of sets in
arbitrary infinite base spaces. The question of the existence of a
maximal (in the sense of cardinality) $\aleph_0$-independent \footnote{We say that a family $(X_i)_{i \in I}$ of subsets of the set $E$
is $\aleph_0$-independent if the condition
$(\forall J)(J \subset I ~ \& ~ \mbox{card}(J)
 < {\aleph_0} \rightarrow \bigcap_{i \in J}
 \overline{X}_i \neq \emptyset )$
holds, where
$(\forall i)(i \in I \rightarrow (\overline{X}_i=X_i) \vee (\overline{X}_i=
(E \backslash X_i))).$  If in addition, this condition holds true  for each $J$ with $card(J) \le \aleph_0$, then $(X_i)_{i \in I}$ is called  strict $\aleph_0$-independent.}family of
subsets of an uncountable set $E$ was considered by A. Tarski. He
proved that this cardinality is equal to ~$2^{card(E)}$.

This result  found an interesting application in general topology.
For example,  it was proved that in an arbitrary infinite space
$E$ the cardinality of the class of all ultrafilters is equal
to~$2^{2^{card(E)}}$~(see, e.g., \cite{KuMos80}).

The combinatorial question of the existence of a maximal (in the sense of cardinality) strict $\aleph_0$-independent 
family of
subsets of a set $E$ with cardinality of the continuum also was investigated and was proved that this cardinality is equal to $2^c$.

This combinatorial result  found  an interesting application in the Lebesgue measure theory. For example,  Kakutani and Oxtoby \cite{KakOxt50}
firstly constructed  a family  $\cal{A}$ of almost invariant subsets of the circle in such a way that
$$
\cap_{n=1}^{\infty}A_n^{\epsilon_n}
$$
has outer measure $1$ for an arbitrary sequence $\{ A_n \}$ of sets from $\cal{A}$ and arbitrary sequence $\{\epsilon_n\}, \epsilon_n=0,1$.  The putting $\overline{m}(A)=1/2$ for $A$ in $\cal{A}$ they obtained an extension of the Lebesgue measure on the circle to an invariant measure $\overline{m}$ such that $L_2(\overline{m})$ has the Hilbert space dimensional equal to $2^c$.

Using the same combinatorial  result, A.B. Kharazishvili constructed a maximal (in the sense of
cardinality) family of orthogonal elementary $D_n$-invariant
extensions of the Lebesgue measure (see \cite{Khar83}).

The combinatorial question of the existence of a
maximal (in the sense of cardinality) strict $\aleph_0$-independent family of
subsets of a set $E$ with $card(E^{\aleph_0})=card(E)$  was investigated  in  \cite{Pan95}  and   it was shown that this cardinality
is equal to $2^{card(E)}$.
Using this result, G.Pantsulaia \cite{Pan89-2} extended  Kakutani and Oxtoby \cite{KakOxt50} method for a  construction of a maximal (in the sense of
cardinality) family of orthogonal elementary $H$-invariant
extensions of the Haar measure defined in a locally compact $\sigma$-compact  topological group with $card(H^{\aleph_0})=card(H)$.

{\bf Method IV.}( Kodaira and Kakutani method) ~   Kodaira and Kakutani \cite{KodKak50} invented the following method of extended the Lebesgue measure on the circle to an invariant measure as follows:

Let produce a {\it character} $\pi$ of the circle, i.e. a homomorphism $\pi : T \to T$ in such a way that the outer Lebesgue measure of its graph $D_{\pi}$ is equal to $1$ in $T \times T$. Then the extended $\sigma$-algebra $\overline{B}$ consists  of sets $A_M=\{ x: (x, \pi(x)) \in M\}$, where $M$ is Lebesgue measurable set in $T \times T$ and the extended measure $\overline{m}$ is $\overline{m}(A_M)=(m \times m)(M).$  Note that the discontinuous character $\pi$ becomes $\overline{B}$-measurable. It has been noticed later in \cite{Hulanicki1959} that one can produce $2^c$ such characters so that they all become measurable and $L^2(\overline{m})$ is of Hilbert space dimension $2^c$.

This method have been modified   for  $n$-dimensional Euclidean space   in  \cite{Pan95}(Received  7.  November  1993)  for a construction of  the invariant extension $\mu$  of the $n$-dimensional Lebesgue measure such that there exists a $\mu$-measurable set with only one density point.  This result answered positively to a certain question  stated by A.B. Kharazishvili (cf. \cite{Khar83}, Problem 9, p. 200). Knowing this result,  A.B. Kharazishvili considered  similar but originally modified method and  extended  previous  result in  \cite{Khar94}( Received 15. March 1994)  as follows:  {\it there exists an invariant extension $\mu$ of the classical Lebesgue measure such that $\mu$ has the uniqueness property and there exists a $\mu$-measurable set with only one density point.}.

{\bf Method  $\star$ }. More lately,  Kodaira and Kakutani method have been modified   for  an uncountable locally compact $\sigma$-compact topological
group $H$ with $card(H^{\aleph_0}) = card(H)$   in  \cite{Pan04-3}  as follows: Let $E$ be  a  set with $2 \le card(E) \le card(H)$ and let $\mu$  be a probability measure in $E$  such that  each $X \in dom(\mu)$ for which $card(X)<card(E)$.    Let produce a {\it function}  $f : H \to E$ in such a way that
the following two conditions

$1)~~ (\forall  e)(\forall F)(e \in E   ~\&~ (F$ is a closed subset
of the $H$  with~$m(F) > 0) \rightarrow
\mbox{card}f^{-1}(e) \cap F)=card(E));$

$2)~~ (\forall E')(\forall g)(E' \subseteq E~ \& ~ g \in H
\rightarrow card(g(\bigcup_{e \in E'}f^{-1}(e)) \triangle (\bigcup_{e
\in E'}f^{-1}(e))) < card(H) ).$

holds true.    Then the extended $\sigma$-algebra $\overline{B}$ consists  of sets $A_M=\{ x: (x,f(x)) \in M\}$, where $M \in dom(m)\times dom(\mu)$.   Then the  extended measure $\overline{m}_{\mu}$  is defined by $\overline{m}_{\mu}(A_M)=(m \times \mu) (M).$  Note that $\overline{m}_{\mu}$ is a non-elementary invariant extension  of the measure $m$  iff  the measure $\mu$  is diffused.  It has been noticed  that   one can produce $2^{card(H)}$   such  functions   so that they all become measurable and $L^2(\overline{m}_{\mu})$ is of Hilbert space dimension $2^{card(H)}$.

Note that when $card(E)=2$, $\mu$ is a normalized counting measure in $E$ and $f :  R^2 \to E$  is defined by $f(x)=\chi_{A}(x)$,  then   Method  $\star$   gives  Marczewski  method.

When $H=E=T$,  $\mu=m$ and $f=\pi$ is  a character, then   Method  $\star$  gives Kodaira and Kakutani method.

Now let us discuss whether Method  $\star$  gives Oxtoby and Kakutani method. In this context we need some auxiliary facts.

\begin{lemma}(  \cite{Pan07},   {\bf Theorem 11.1} , p. 158   )   If an infinite set $E$ satisfies the
condition
$$Card(E^{\aleph_0})=card(E),$$
 then there exists a
maximal (in the sense of cardinality) strictly $\aleph_0$-independent
family $(A_i)_{i \in I}$ of subsets of the space $E$, such that
$$card(I)=2^{card(E)}.$$
\end{lemma}

\begin{lemma}(  \cite{Pan07},   {\bf Lemma  11.2} , p. 163   )  Let $H$ be an arbitrary locally compact
$\sigma$-compact topological group, $\lambda$ be the Haar measure
defined on the group $H$ and let $\alpha$ be an arbitrary cardinal
number such that:
$$\alpha \leq {card(H)}.$$

Then there exists a family $(X_i)_{i \in I}$ of subsets of the set
$H$ such that:

$1)~~ Card(I)= \alpha;$

$2)~~ (\forall i)(\forall i')(i \in I \ \& \ i' \in I \ \& \ i
\neq i' \rightarrow X_i \cap X_{i'}=\emptyset );$

$3)~~ (\forall i)(\forall F)(i \in I ~\&~ (F$ is a closed subset
of the space $H$ with~$\lambda(F) > 0) \rightarrow$

$\mbox{card}(X_i \cap F)=\mbox{card}(H));$

$4)~~ (\forall I')(\forall g)(I' \subseteq I \ \& \ g \in H
\rightarrow card(g(\bigcup_{i \in I'} X_i) \triangle (\bigcup_{i
\in I'} X_i)) < card(H)).$
\end{lemma}

\begin{lemma}(  \cite{Pan07},   {\bf Lemma  11.5} , p. 169   ) Let $E$ be an uncountable base space with
$card(E^{\aleph_0})= card(E)$. Then there exists a non-atomic
probability measure $\cal P$ such that the following conditions
hold:

$a)~~~ (\forall X)(X \subseteq E \ \& \ card(X) < card(E)
\rightarrow {\cal P}(X)=0);$

$b)~~$ the topological weight $a({\cal P})$ of the metric space
$(\mbox{dom}({\cal P}),{\rho}_{{\cal P}})$ associated with measure
${\cal P}$ is maximal, in particular, is equal to $2^{card(E)}$.
\end{lemma}

Now put  $H=G$ and  $ E=\{0,1\}^{N}$.  For $g \in G$, we put $h(g)=(g,i)$, where $i$ is a unique index for which $g \in X_i$(cf. Lemma 4.2 for $\alpha=card(H)=2^{\aleph_0}$). We 
 set $\overline {A}_i=\cup_{ e \in  X_i}  f^{-1}(e)$  for $i \in  I$. Let $ {\cal P}$ comes from  4.3.  Now if we consider the invariant extension $\overline{m}_{ {\cal P}}$
 of the measure $m$, we observe that $\overline{m}_{ {\cal P}}(\overline {A}_s)=1/2$ for each $\overline {A}_s \in \cal{A}:= \{\overline {A}_i : i \in I\}$.  Since $\cal{A}$ is the family of strictly $\aleph_0$-independent almost invariant subsets of $G$, we claim that   Method  $\star$    gives just  above described Oxtoby and Kakutani method.

{\bf Method V.} ( Kharazishvili ).  This approach, as usual,  can be used for uncountable commutative groups and  is  based on purely algebraic properties those groups, which are  not assumed to be endowed with any topology but only are equipped with a nonzero $\sigma$-finite invariant
measures.  Here essentially  is used Kulikov's well known theorem about covering of any commutative group by  increasing (in the sense of  inclusion) countable sequence of subgroups of $G$ which are  direct sum of cyclic groups (finite or infinite) (see, for example \cite{Khar05}, \cite{Khar09}).

{\bf Definition 4.4(Krarazishvili)} Let $E$ be a base space, $G$ be a group of
transformations of $E$ and let $X$ be a subset of the space $E$.
$X$ is called  a  $G$-absolutely negligible set if for any
$G$-invariant $\sigma$-finite measure $\mu$, there exists its $G-$invariant extension $\Bar{\mu}$  such that $X \in
\mbox{\mbox{dom}}(\Bar{\mu})$ and $\Bar{\mu}(X)=0$.
\medskip

A geometrical characterization of absolutely negligible subsets,
due to A.B. Kharazishvili, is presented in the next proposition.
\medskip

{\bf Theorem 4.5} {\it Let $E$ be a base space, $G$ be a group of
transformations of $E$ containing some  uncountable subgroup
acting freely  in  $E$, and $X$ be  an arbitrary subset of the
space $E$. Then the following two conditions are equivalent:

1)~~ X is a $G$-absolutely negligible subset of the space $E$;

2)~~ for an arbitrary countable $G$-configuration \footnote{A subset $X'$ of $E$ is called a countable $G$-configuration of $X$ if there 
is a countable family $\{g_k : k \in N\}$ of elements of $G$ such that $X' \subseteq \cup_{k \in N} g_k(X)$.} $X'$ of the set
$X$, there exists a countable  sequence $(h_k)_{k \in N}$ of
elements of $G$ 

$$\bigcap_{k \in N} h_k(X')=\emptyset.$$}

It is of interest that the class  of all countable $G$-configurations of the fixed  $G$-absolutely negligible subset  constitutes  a $G$-invariant $\sigma$-ideal such that the inner measure of each element of this class is zero with respect to any $\sigma$-finite $G$-invariant measure in $E$. Hence, by using the natural modification of the Method I one can obtain $G$ -invariant extension of an arbitrary   $\sigma$-finite $G$-invariant measure in $E$.

In 1977 A. B. Kharazishvili  constructed   the partition  of the real axis  $R$  in to  the countable  family of  $D_1$ -absolutely  negligible sets  and  got   the negative  answer to the question   of   Waclaw Sierpi\'{n}iski  in the one-dimensional  case without any set-theoretical assumption (see  \cite{Khar77}).

Finally, in 1982 Krzysztof Ciesielski
and Andrzej Pelc generalized Kharazishvili's result to all
$n$-dimensional Euclidean spaces, more precisely,  they   constructed   the partition  of the Euclidean  space   $R^n$  in to  the countable  family of  $D_n$ -absolutely  negligible sets  and  got   the negative   answer  to the question   of     Waclaw Sierpi\'{n}iski   in the $n$-dimensional  case without any set-theoretical assumption (see \cite{Cies85}).

By using the method of absolutely negligible sets elaborated by A.Kharazishvili  \cite{Khar83},  P. Zakrzewski  \cite{Zakrzewski95} answered positively to a question of  Ciesielski \cite{Cies91}  asking {\it whether an isometrically invariant $\sigma$-finite countably additive measure on $R^n$ admits  a strong countably additive isometrically invariant extension.} It is obvious that this question is generalization of the  above  mentioned    Waclaw Sierpi\'{n}iski  problem.

\section{Discussion}

Let $G$ be a compact Hausdorff topological group and $\lambda$ be a Haar measure in $H$. By
$~{B}(G)$  we mean the set of all bounded real-valued Borel-measurable functions
on $G$. Under the norm $||f||=\sup_{g \in G}|f(g)|$ for $f \in ~{B}(G)$, the set $~{B}(G)$
forms a Banach space, and even a Banach algebra if algebraic operations
for functions are defined in the usual way. The subset  $~{R}(G)$ of  $~{B}(G)$ consisting
of all real-valued continuous functions on $G$ is then a Banach subalgebra
of $~{B}(G)$.

Following \cite{KuNi74}(see Definition 1.1, p.171), The sequence $(x_n)_{n \in N}$ of elements in  $G$  is
called $\lambda$-u.d. in $G$ if
$$
\lim_{n \to \infty}\frac{\sum_{k=1}^nf(x_k)}{n}=\int_{G} f(x)d \lambda(x)\eqno (5.1)
$$
for all $f \in ~{R}(G)$.

Note that the theory of uniform distribution is well developed in compact Hausdorff topological groups (see, for example \cite{KuNi74}, Chapter 4) as well the theory of invariant extensions of Haar measures in the same groups(see Section 4). Here naturally arise a question asking {\it whether can be introduced the concept of uniform distribution for invariant extensions of the Haar measure in compact Hausdorff topological groups.} We wait that by using similar manner used in Section 3 and the methodology  briefly  described in Section 4, one can resolve this question.

\medskip
\noindent {\bf Acknowledgment}. The authors wish to thank the referees for their constructive critique of the first draft.


\begin{thebibliography}{99}






\bibitem{Cies85} K. Ciesielski and A. Pelc, {\em Extensions of invariant
measures on Euclidean spaces,} Fund. Math. {\bf 125(1)}
(1985), 1-10.

\bibitem{Cies91}{
K. Ciesielski,}  {\em Query 5,} Real Analysis Exchange, {\bf 16(1)} (1990-91), 374.



\bibitem{KakOxt50} {S. Kakutani, J. Oxtoby,} {\em Construction of non--separable invariant extension
of the Lebesgue measure space}, Ann. of  Math., {\bf 52(2)}(1950), 580--590.


\bibitem{Khar77} {A.B. Kharazishvili,}  On Sierpifiski's problem concerning
strict extendability of an invariant measure,
Soviet. Math. Dokl. {\bf 18(1)} (1977), 71-74.



\bibitem{Khar83} {A.B. Kharazishvili,} {\em Invariant extensions of the Lebesgue measure}, Tbilisi,
1983 (in Russian).




\bibitem{Khar94} { A.B. Kharazishvili,}{\em  Some remarks on density points and the uniqueness property for invariant extensions of the Lebesgue measure.} Acta Univ. Carolin. Math. Phys., {\bf  35(2)} (1994),  33--39.

\bibitem{Khar96} { A.B. Kharazishvili,} {\em Selected topics of point set theory}, Wydawnictwo
Uniwersytetu Lodzkiego. Lodz, 1996.


\bibitem{Khar05} { A.B. Kharazishvili,}~  {\em On absolutely negligible sets in
uncountable solvable groups}, Georgian Math. J., {\bf
12(2)} (2005), 255--260.


\bibitem{Khar09} A.B. Kharazishvili, {\em On thick subgroups of uncountable $\sigma$-compact locally compact commutative groups,} Topology Appl. {\bf 156(14)} (2009),  2364--2369.



\bibitem{KodKak50} {K. Kodaira, S. Kakutani,} {\em A nonseparable tranlation--invariant extension
of the Lebesgue  measure space}, Ann. of Math., {\bf 52}
(1950), 574--579.

\bibitem{KuNi74} { L. Kuipers, H. Niederreiter, } {\em Uniform distribution of sequences,}
Wiley-Interscience [John Wiley \& Sons], New York-London-Sydney, 1974.

\bibitem{KuMos80}{ K. Kuratowski, A. Mostowski,} {\em Set theory}, Nauka, Moscow , 1980.
(in Russian).


\bibitem{Pan89-2}{G.R. Pantsulaia,} {\em Independent families of sets and some
of their applications to measure theory}, Soobshch. Akad. Nauk Gruzin. SSR  {\bf 134(1)} (1989),
29--32 (in Russian).

\bibitem{Pan94}{G.R. Pantsulaia,} {\em  On non--elementary extensions of the Haar measure},
Reports of Enlarged Sessions of the Seminar of I.Vekua
Inst. Appl. Math., Tbilisi, {\bf 9(13)} (1994), 40--43.


\bibitem{Pan95}{G.R. Pantsulaia,}{\em Density points and invariant extensions of the Lebesgue measure}(in Russian),
 Soobshch. Akad. Nauk Gruzii {\bf 151(2)} (1995), 216--219.


\bibitem{Pan04-3}{G.R. Pantsulaia,} {\em An applications of independent
families of sets to the measure extension problem,} Georgian Math. J.,{\bf 11(2)}(2004), 379--390.



\bibitem{Pan07}{G.R. Pantsulaia ,}  {\em  Invariant and quasiinvariant measures in infinite-dimensional topological vector spaces.} Nova Science Publishers, Inc., New York, 2007. xii+234 pp.


\bibitem{Pkhakadze1958}{S.S. Pkhakadze,} {\em K teorii lebegovskoi mery,} Trudy
Tbilisskogo Matematiceskogo Instituta 25, Tbilisi,  1958 (in Russian).


\bibitem{Shiryaev1980}
 A.N. Shiryaev, {\em Probability} (in Russian), Izd.“Nauka”, Moscow, 1980.


\bibitem{Solovay1970} R. Solovay, {\em A model of set theory in which every
set of reals is Lebesgue measurable,}  Ann. of Math. {\bf 92}
(1970), 1-56.



\bibitem{Szpil46}{ E. Szpilrajn ( E. Marczewski),} {\em On problems of the theory of
measure,} {\em Uspekhi Mat. Nauk,} {\bf 1} no. {\bf 2(12)},
(1946),179--188 (in Russian).



\bibitem{Weyl}
 H. Weyl, {\em  \'{U}ber ein Problem aus dem Gebiete der diophantischen Approximation.}  Marchr. Ges. Wiss. G\'{o}tingen. Math-phys. K1.(1916), 234-244



\bibitem{Hardy2} G. Hardy, Littlewood J.,  {\em Some problems of
diophantine approximation.}  Acta Math.{\bf 37 (1)} (1914), 193--239.



\bibitem{Hardy1}
 G. Hardy, J. Littlewood, {\em Some problems of diophantine
approximation.}  Acta Math. {\bf 37 (1)}(1914), 155--191.


\bibitem{Hlawka56}
E. Hlawka, {\rm Folgen auf kompakten Raumen,}  Abh. Math. Sent.
Hamburg {\bf 20} (1956), 223-241.


\bibitem{Hulanicki1959}
Hulanicki, A. {\em On subsets of full outer measure in products of measure spaces.} Bull. Acad. Polon. Sci. S�r. Sci. Math. Astr. Phys. {\bf 7 } (1959),  331--335.



\bibitem{Hulanicki1962}
 A. Hulanicki,  {\em Invariant extensions of the Lebesgue measure.}  Fund. Math. {\bf 51}  (1962/1963),  111--115.


\bibitem{HulanickiRyll-Nardzewski1979}
 A. Hulanicki, C. Ryll-Nardzewski, {\em Invariant extensions of the Haar measure.}  Colloq. Math. {\bf 42}  (1979), 223--227.


\bibitem{Zakrzewski95}
P. Zakrzewski, {\em  Extending isometrically invariant measures on ${\bf R}\sp n$  -- a solution to Ciesielski's query.}  Real Anal. Exchange {\bf 21(2)} (1995/96), 582--589.


 \end{thebibliography}
\end{document}